\documentclass[12pt]{article}
\usepackage[utf8]{inputenc}
 \usepackage{graphicx}
 \usepackage[T2A]{fontenc}
 \usepackage[russian,english]{babel}
 \usepackage{amsmath,amssymb,euscript,amsthm,amsfonts,mathrsfs,amscd}
 \usepackage{color}
 \usepackage{floatflt}
 \usepackage{mathrsfs} 

\title{Abstract, keywords and references template}
\author{Author Surname$^{1}$, Someone Else$^{2}$  \\
        \small $^{1}$University A \\
        \small $^{2}$University B \\
}
 
 \usepackage[ left=20mm, top=20mm, right=20mm, bottom=20mm, footskip=10mm, nohead, nomarginpar, driver=xetex]{geometry}
\usepackage{amsthm}
\theoremstyle{plain}
\newtheorem{thm}{Theorem}
\newtheorem{lem}{Lemma}
\newtheorem{prop}{Proposition}
\newtheorem{df}{Definition}

\theoremstyle{definition}
\newtheorem{rem}{Remark}
\newtheorem{exmp}{Example}

\newcommand{\gr}{\color{green}}
\newcommand{\bb}{\color{blue}}

\newcommand{\dd}{\ensuremath{\displaystyle}}

\newcommand{\infl}{\ensuremath{\inf\limits}}
\newcommand{\supl}{\ensuremath{\sup\limits}}
\newcommand\intl{\ensuremath{\int\limits}}
\newcommand{\suml}{\ensuremath{\sum\limits}}
\newcommand{\bD}{\stackrel{\mathscr{D}}{=}}
\newcommand{\1}{\ensuremath{\mathbf{1}}}
\newcommand{\PPP}{\mathscr{P}}
\newcommand{\FFF}{\mathscr{F}}

\newcommand{\BBB}{\mathscr{B}}
\newcommand{\ZZZ}{\mathscr{Z}}

\newcommand{\UU}{\mathscr{U}}

\newcommand{\XX}{\mathscr{X}}
\newcommand{\ud}{\,\mathrm{d}}

\newcommand{\bd}{\stackrel{\text{\rm def}}{=\!\!\!=}}
\newcommand{\PP}{\ensuremath{\mathbf{P}}}

\newcommand{\EE}{\ensuremath{{\mathbb{E}}}}

\newcommand*{\TR}{\hfill \ensuremath{\triangleright}}

\title{About quasi-renewal processes and quasi-regenerative processes 
}
\author{Galina A. Zverkina
\\ V.~A.~Trapeznikov Institute of Control Sciences
\\of Russian Academy of Sciences,\\
Institute for Information Transmission Problems\\ 
 of Russian Academy of Sciences (A.~A.~Kharkevich Institute),\\ Russia
}

\begin{document}
\maketitle
\begin{abstract}

The study of the behaviour of stochastic processes in queuing theory and related fields is often based on the study of the behaviour of regenerative Markov processes.
Often these regenerative processes are a combination of piecewise linear stochastic processes.
We discuss  the definition of a piecewise linear process which is comfortable to study the complex stochastic models currently being studied.

In many cases, a piecewise linear Markov process has an embedded renewal process, and hence the study of the behaviour of a piecewise linear Markov process can be based on an analysis of the behaviour of this embedded renewal process.

But the behaviour of a complex queuing system or a complex reliability model can be described by a non-regenerative stochastic process, but this non-regenerative process can be in some sense ``close'' to some regenerative process.
In this case, embedded processes that would correspond to embedded renewal processes turn out to be in some sense ``close'' to some ``classical'' renewal processes.

Therefore, in this paper, we introduce the concept of quasi-renewal and quasi-regenerative processes.
And we describe an example of such a stochastic model whose behaviour is described by a quasi-regenerative process, and an example of a stochastic model having embedded quasi-renewal processes.

We also propose a method for obtaining the upper bounds for the convergence rate of the distribution of a regenerative and quasi-regenerative process to a stationary distribution, if this process is ergodic.
\end{abstract}


\section{Introduction}
As is well known, most problems in queuing theory and related fields are based on renewal theory.
Usually, it is the renewal processes that describe the process of input flow of customers in a queuing system, their service process, or a sequence of failures or repairs in reliability theory. 
In queuing networks, also as a rule, renewal processes allow one to describe the periods of stay of the customers at network nodes.

In complex systems consisting of several operating units, the behaviour of each unit can be described by the renewal process (with individual characteristics).
However, the state of these units can influence each other.
Accordingly, the characteristics of the renewal processes corresponding to the behaviour of these units may change.
If such changes are small-scale in some sense, it becomes necessary to study such quasi-renewal processes.

Moreover, a great number of processes in queuing theory are regenerative, i.e. they get into a certain state at random intervals, which are almost surely finite, and after they hit this state, the process is ``restarted''.
Hence, the regenerative process has an embedded renewal process, with renewal periods equal to the length of the period of regeneration (the interval between two consecutive hits in the regeneration state).

But in some situations, the studied process is not regenerative, but it is ``close'' to the regenerative process in some sense.
Below will be studied one type of such a process close to the regenerative process.

\section{Renewal processes, piecewise linear processes and regenerative processes}
\subsection{Renewal processes}
Recall the definition of the renewal process and let's discuss in what sense we will use this concept.
\begin{df}\label{defRen}
The renewal process $N_t$ is a counting process \linebreak $N _t\bd\dd\suml _{i=1} ^\infty \1\left \{ \suml _{s=1} ^i \xi _k\leqslant t\right \} $, where $\left \{\xi _1, \xi _2, ...\right \} $ are i.i.d. positive random variables.
$N _t$ changes its states at the times $t _k=S _k\bd\dd\suml _{j=1} ^s \xi _j$.
The times $t _k$ are renewal times.
\TR
\end{df}

The behaviour of the renewal process is determined by the distribution function (d.f.) $F(t)=\PP\{\xi_k\leqslant t\}$ of its regeneration period $\xi_k$.

In some situations, we are interested in the time from the last renewal of the process $N_t$ and the current time $t$; also we can be interested in the time, and how long to wait for the next renewal from the current time.

\begin{df}
Consider some renewal process $N_t$.
For any time $t\geqslant{0}$ denote $B_t\bd t-S_{N_t}$ -- a backward renewal time (or overshot) of the renewal process $N_t$ at the time $t$.
Also, for any time $t\geqslant{0}$ denote $W_t\bd S_{N_t+1}-t$ -- a forward renewal time (or undershot) of the renewal process $N_t$ at the time $t$.   
\TR
\end{df}

If   at the time $t\geqslant 0$ the value of $B_t$  is known ($B_t=a$), then the d.f. of $W_t$ is also known: $F^W(s)=\PP\{W_t\leqslant s\}=\dd\frac{F(a+s)-F(a)}{1-F(a)}$.

\subsection{Backward renewal time of a renewal process as a Markov process}
In a classical sense, renewal processes are counting processes.
However, in addition to the number of renewals in applications of renewal theory, it is also important to take into account the random duration of renewal periods.

These renewal periods are i.i.d. r.v.'s $\{\xi_i\}$ with d.f. $F(s)\bd \PP\{\xi_i\leqslant s\}$.

If the backward renewal time $B_t\bd t-\suml_{i=1}^{N_t}\xi_i=a$, then the next renewal can be in the interval $(t,t+\Delta]$ ($\Delta \ll 1$) with probability $\PP_{a;(t,t+\Delta]}=\dd\frac{F(a+\Delta)-F(a)}{1-F(a)}=\frac{F'(a)\Delta}{1-F(a)}+o(\Delta)$ assuming absolute continuity of d.f. $F(s)$.

\begin{df}\label{lambda}
For the d.f. $F(t)$ of non-negative r.v. the function $\lambda(s)\bd \dd \frac{F'(s)\Delta}{1-F(s)}$ is called the intensity function for the distribution $F(\cdot)$.\TR   
\end{df}

\begin{rem}
In reliability theory and some papers on the queuing theory, the intensity of renewal processes (concerning the nature of the occurrence of the renewal process in an applied problem) is often called ``hazard rate'', ``failure rate'', input flow intensity, etc.

But in the renewal theory, abstracting from the nature of the origin of the renewal process, it is natural to call this value ``intensity'' (of the end of the renewal period).
\TR
\end{rem}

Emphasize, 
\begin{equation}\label{F}
F(s)=1-\exp\left(-\intl_0^s \lambda(u)\ud u\right),    
\end{equation}
and the distribution of the renewal period can be determined by intensity $\lambda(s)$ as well as by d.f. $F(s)$.

Also, for the mixed (non-singular) distributions we can use a generalized notion of intensity.
\begin{df}\label{INT}
For non-singular d.f. $F(s)$ put 
$
f(s)=
\begin{cases}
F'(s), & \mbox{if} F'(s) \mbox{ exists;} \\
0, & \mbox{in the other case,}
\end{cases}
$ 
\\
and the generalized intensity of d.f. $F(s)$ is
$$\lambda(s)\stackrel{{\rm def}}{=\!\!\!=}  \displaystyle \frac{f(s)}{1-F(s)}-\sum\limits _{i}\delta(s-a_i)\ln\big(F(a_i+0)-F(a_i-0) \big),
$$
where  $\{a_i\}$ --- is the set of all points of discontinuity of a function $F(s)$, and $\delta(\cdot)$ is a standard $\delta$-function.
\TR
\end{df}
\begin{rem}
The formula \ref{F} remains correct (see \cite{KalimulinaZverkinaICSM2020}).
\TR
\end{rem}


Thus, at any time $t$ the next behaviour of the backward renewal time is defined by the intensity: it gives the probability of a transition to the state zero (renewal time) or a subsequent linear increase of the value of the backward renewal process.
So, it is Markov.

\begin{df}
A generalized non-negative function $\lambda(s)$ defined on $[0;\infty)$ such that 
\begin{equation}\label{usl_int}
    \dd\intl_0^\infty \lambda(s)\ud s=+\infty
\end{equation}
is called a (generalized) intensity function. 
\TR
\end{df}
\begin{rem}
It is easy to see that the (generalized) intensity function determines the d.f. $F(s)$ according to formula (\ref{F}).
So, the distribution of a non-negative random variable can be determined using the d.f. $F(s)$, distribution density $f(s)\bd F'(s)$, and intensity function (see \cite{KalimulinaZverkina19}).
\TR
\end{rem}
\begin{rem}
    In what follows, when discussing renewal processes, we will be primarily interested in the behaviour of the backward renewal time $B_t$.
    This is because we are interested in the asymptotic behaviour of renewal processes, and not in the total number $N_t$ of renewals, which tends to infinity.
    \TR
\end{rem}

\subsection{Piecewise linear processes}

If we observe several interacting renewal processes, then we may be interested in when the next renewal of at least one of the processes will occur.
Such situations often arise in the study of complex queuing systems or queuing networks or reliability systems.

Probably the first definition of piecewise linear process was proposed by \cite{Belyaev62}.
There the piecewise linear process is defined as two-dimensional Markov processes $\{\nu(t), u(t)\}$, where $\nu(t)$ is a semi-Markov process, and $u(t)\bd t-\sup\{u:\; \nu(u)\neq  \nu(t)\}$, i.e. $u(t)$ is the time that has elapsed since the last transition of the semi-Markov process $\nu(t)$.
The graph of the process defined in this way is shown (1st variant) in Fig.\ref{fig3}.
This process has a single piecewise linear component, which changes at a unit rate.

Also, the piecewise linear process can be defined as a process $\{\nu(t), v(t)\}$, where $v(t) = \inf\{u>t:\;  \nu(u)\neq  \nu(t)\} - t$, i.e. $v(t)$ is the time remaining until the next transition of $\nu(t)$.

In both these definitions, the processes $\{\nu(t), u(t)\}$ and $\{\nu(t), v(t)\}$ are Markov (see, e.g., \cite{Koroluk75}).

Note that the definition given in \cite{Belyaev62, Koroluk75} describes a process with a single piecewise continuous component. 
It is convenient, for example, to describe service in a queuing network, where $\nu(t)$ indicates the node number, and service occurs at a constant unit rate.

\begin{figure}[h]
\begin{center}
\begin{picture}(320,75)
\put(0,20){\vector(1,0){320}}
\put(0,18){\line(0,1){4}}
\thicklines
\put(0,20){\line(1,1){40}}
\multiput(0,18)(0,5){10}{\line(0,1){3}}
\multiput(40,18)(0,5){10}{\line(0,1){3}}
\put(40,20){\line(1,1){60}}
\multiput(100,18)(0,5){13}{\line(0,1){3}}
\put(100,20){\line(1,1){40}}
\multiput(140,18)(0,5){10}{\line(0,1){3}}
\multiput(170,18)(0,5){7}{\line(0,1){3}}
\put(140,20){\line(1,1){30}}
\put(170,20){\line(1,1){40}}
\multiput(210,18)(0,5){9}{\line(0,1){3}}
\put(210,20){\line(1,1){58}}
\multiput(288,30)(6,0){6}{$\cdot$}
\put(0,5){\scriptsize$t_0$}
\put(40,5){\scriptsize$t_1$}
\put(100,5){\scriptsize$t_2$}
\put(140,5){\scriptsize$t_3$}
\put(170,5){\scriptsize$t_4$}
\put(210,5){\scriptsize$t_5$}
\put(0,70){\scriptsize$\nu(t)= \{i_1\}$}
\put(45,70){\scriptsize$\nu(t)= \{i_2\}$}
\put(101,70){\scriptsize$\nu(t)= \{i_3\}$}
\put(145,50){\scriptsize$\nu(t)= \{i_4\}$}
\put(170,70){\scriptsize$\nu(t)= \{i_5\}$}
\put(215,75){\scriptsize$\nu(t)= \{i_6\}$}
\end{picture}
\end{center}
\caption{Visualisation of the piesewise process following \cite{Belyaev62}.}
\label{fig3}
\end{figure}
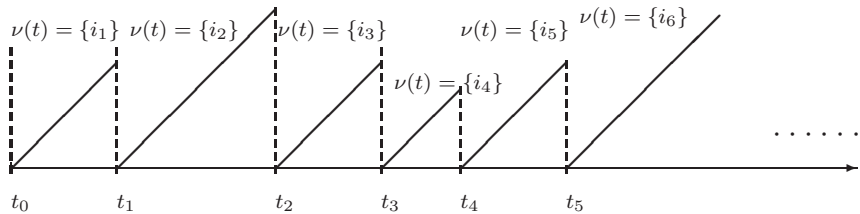

Another slightly more general definition of a process is given in \cite{Kuczura73}.
This definition has been adapted for use in the study of queuing systems of the form $M|G|1$, $GI|M|s$ and $GI + M|M|s$.

Note that the definitions (\cite{Belyaev62, Koroluk75, Kuczura73}) mentioned here describe a process with a single piecewise continuous component.
They are convenient, for example, to describe the behaviour of one-line queuing systems, one-element pressure models, and also queuing networks, where $\nu$ indicates the node number, and service occurs at a constant unit rate.

A much more general definition of a piecewise linear process was given in \cite[\S 3.3.2]{GnedenkoKovalenko89}, let's call this definition ``classical'':
\begin{df}\label{defLin}
Piecewise linear random process $X_t = (\vec\nu(t), \vec\xi(t))$ defined as follows.

1. The space of states $\XX$ of the process $X_t$ is the set of pairs $(\vec\nu,\vec \xi_\nu)$, where $\vec\nu$ is the element of a finite or countable set. 
$\vec\xi_\nu$ is the vector $(\xi_1, \ldots \xi_{|\nu|})$, $|\nu|\geqslant 0$ is the ``rank'' (or the index number of the ``basic state'' of the process $X_t$) of the state $\nu$, and $\xi_j\geqslant O$ (all possible ``base states'' are numbered in some order).

2. Let $X_t = (\vec\nu,\vec y),$ $\vec y = (y_1, \ldots y_{|\nu|})$.
With probability $\lambda_{\nu}(X_t)\times p_{\nu,\nu'} \ud t$ during time $\ud t$ a spontaneous transition $X_t$ to the state $\vec\nu'$ takes place.
After the transition a new value of $ \vec\xi(t)$ is random and possesses a measurable in $\vec y$ distribution function
$$
B_{\nu,\nu'}^{(0)}(\vec x| \vec y) =\PP\{ \vec\xi(t + \ud t) < \vec x | \vec\xi(t) = \vec y, \vec\nu(t + \ud t) = \vec\nu'\}
$$
The probability is $o(h)$ that two or more spontaneous transitions will occur during a short period $h$.

3. In the absence of spontaneous transitions in the interval $(t, t + \ud t)$, we have $\vec\nu(t + \ud t) = \vec\nu(t)$, $\vec\xi(t + \ud t) =\vec\xi(t) - \vec \alpha_\nu \ud t$, where $\vec \alpha_\nu= (\alpha_{\nu,1}, \ldots, \alpha_{\nu,|\nu|})$ is a vector with non-negative components.
    \TR
\end{df}
\begin{rem}
Here and below, it is assumed everywhere that all components of the index vectors are arranged in ascending order.
\TR
\end{rem}

\begin{rem}
From a modern point of view, the Definition \ref{defLin} is not strict.
However, it allows one to describe the behaviour of various queuing systems and queuing networks.
The definition of a piecewise linear process has changed over time.
\TR
\end{rem}
Following the Definition \ref{defLin}, at the time after the change of the component $\nu$ (``rank'') of the process $X_t$, all components of the vector $\vec \xi_\nu$ are positive, and they decrease (with some known speed) until the next change of the ``rank'' $\nu$ -- and the ``rank'' must be changed at the time when at least one of the components of the vector $\xi_\nu$ decreases to the value 0.

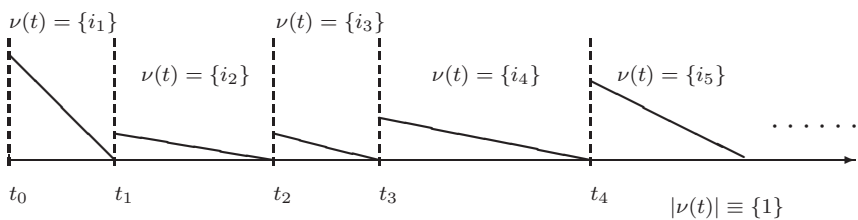
\begin{figure}[h]
\begin{center}
\begin{picture}(320,90)
\put(0,20){\vector(1,0){320}}
\put(0,18){\line(0,1){4}}
\thicklines
\put(0,60){\line(1,-1){40}}
\multiput(0,18)(0,5){10}{\line(0,1){3}}
\multiput(40,18)(0,5){10}{\line(0,1){3}}
\put(40,30){\line(6,-1){60}}
\multiput(100,18)(0,5){10}{\line(0,1){3}}
\put(100,30){\line(4,-1){40}}
\multiput(140,18)(0,5){10}{\line(0,1){3}}
\put(140,36){\line(5,-1){80}}
\multiput(220,18)(0,5){10}{\line(0,1){3}}
\put(220,50){\line(2,-1){58}}
\multiput(288,30)(6,0){6}{$\cdot$}
\put(0,5){\scriptsize$t_0$}
\put(40,5){\scriptsize$t_1$}
\put(100,5){\scriptsize$t_2$}
\put(140,5){\scriptsize$t_3$}
\put(220,5){\scriptsize$t_4$}
\put(250,00){\scriptsize$|\nu(t)|\equiv \{1\}$}
\put(0,70){\scriptsize$\nu(t)= \{i_1\}$}
\put(50,50){\scriptsize$\nu(t)= \{i_2\}$}
\put(101,70){\scriptsize$\nu(t)= \{i_3\}$}
\put(160,50){\scriptsize$\nu(t)= \{i_4\}$}
\put(230,50){\scriptsize$\nu(t)= \{i_5\}$}
\end{picture}
\end{center}
\caption{The graph of an example of a one-dimensional piecewise linear process (meaning the dimension of a piecewise continuous component of the process).}
\label{fig1}
\end{figure} 

An example of a one-dimensional piecewise linear process is shown in Fig. \ref{fig1}. 
This graph can visualize, for example, the movement of a customer through a network.
Here it is assumed that at the moment a customer arrives at the network node, the scope of service work and the rate of service are known.

\begin{figure}[h]
\begin{center}
\begin{picture}(320,120)
\put(0,20){\vector(1,0){320}}
\put(0,18){\line(0,1){4}}
\thicklines
\put(0,40){\line(2,-1){25}}
\put(25,28){\line(4,-1){30}}
\put(55,50){\line(4,-1){15}}
\put(70,46){\line(2,-1){29}}
\put(99,31){\line(4,-1){40}}
\put(190,35){\line(2,-1){30}}
\put(220,50){\line(2,-1){20}}
\put(240,40){\line(4,-1){50}}
\multiput(288,30)(6,0){6}{$\cdot$}
\put(0,80){\vector(1,0){320}}
\put(0,78){\line(0,1){4}}
\thicklines
\put(25,103){\line(2,-1){45}}
\multiput(0,18)(0,5){20}{\line(0,1){3}}
\multiput(25,18)(0,5){20}{\line(0,1){3}}
\multiput(55,18)(0,5){18}{\line(0,1){3}}
\multiput(70,18)(0,5){20}{\line(0,1){3}}
\multiput(100,18)(0,5){20}{\line(0,1){3}}
\multiput(140,18)(0,5){20}{\line(0,1){3}}
\multiput(155,18)(0,5){20}{\line(0,1){3}}
\multiput(190,18)(0,5){20}{\line(0,1){3}}
\multiput(220,18)(0,5){20}{\line(0,1){3}}
\multiput(240,18)(0,5){20}{\line(0,1){3}}
\put(100,115){\line(2,-1){40}}
\put(140,95){\line(1,-1){15}}
\put(155,115){\line(1,-1){35}}
\put(190,110){\line(1,-1){30}}
\put(240,110){\line(2,-1){50}}
\multiput(288,90)(6,0){6}{$\cdot$}
\put(0,0){\scriptsize$t_0$}
\put(25,0){\scriptsize$t_1$}
\put(55,0){\scriptsize$t_2$}
\put(70,0){\scriptsize$t_3$}
\put(100,0){\scriptsize$t_4$}
\put(140,0){\scriptsize$t_5$}
\put(155,0){\scriptsize$t_6$}
\put(190,0){\scriptsize$t_7$}
\put(220,0){\scriptsize$t_8$}
\put(240,0){\scriptsize$t_9$}
\put(-25,115){\scriptsize$\scriptsize\nu(t):$}
\put(5,115){\scriptsize$\{1\}$}
\put(40,115){\scriptsize$\{1,2\}$}
\put(80,115){\scriptsize$\{1\}$}
\put(110,115){\scriptsize$\{1,2\}$}
\put(140,115){\scriptsize$\{2\}$}
\put(170,115){\scriptsize$\{2\}$}
\put(195,115){\scriptsize$\{1,2\}$}
\put(225,115){\scriptsize$\{1\}$}
\put(250,115){\scriptsize$\{1,2\}$}
\end{picture}
\end{center}
\caption{Graph of an example of a variable-dimensional piecewise linear process.}
\label{Fig2x}
\end{figure}
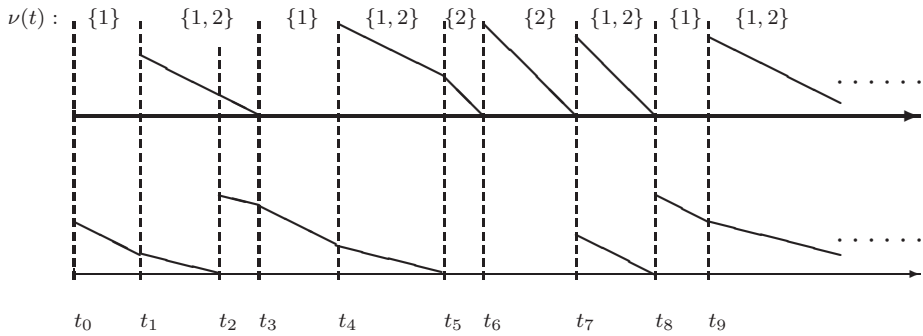

Another example is shown in Fig. \ref{Fig2x}. 
This graph can be interpreted as the behaviour of two elements operating at random times.
They can turn on again after the end of work, or after a while.
Again, the random workload is known at the moment the element is turned on, but the rate of work depends on how many elements are working at a given time.
Here the component $\nu(t)$ can take values $\big\{\{1\},\{2\},\{1,2\}\big\}$. 

In models of queuing theory and related problems, as a rule, the ``rank'' $\nu$ changes only at the moment when at least one of the components of the vector becomes equal to zero.
So, at the time of the change of the ``rank'', we know the time of the next change of the ``rank'', i.e. after a ``rank'' switch, the time of the next ``rank'' switch is known.
This assumption is inconvenient when dealing with renewal processes, namely renewal processes are the main tool in queuing theory and related fields.

But for the asymptotic analysis of the distribution of the stochastic process, it is much more convenient to analyze Markov processes, since there are many methods for studying the behaviour of such processes. 

\begin{df}\label{defMar}
Piecewise linear process random process $X_t = \left(\vec\nu(t); \vec\xi(t)\right)$ defined as follows.
\\
1. The process $X_t$ is 
the set of pairs $\left(\vec\nu(t); \vec\xi(t)\right)=(\nu_1,\nu_2,\ldots,\nu_k;\xi_1,\xi_2,\ldots,\xi_k;)\in\XX$, where process state space is $\XX=\bigcup\limits_{k=1}^\infty\left( \mathbb Z ^k \times   \mathbb R_{\geqslant 0}^k\right)$.
Here $|{\vec\nu}|\bd k> 0$ is the ``rank'' (or the number of the ``basic states'' of the process $X_t$ at the time $t$).
\\
2. 
If we delete components with numbers $i_1,i_2,\ldots,i_k$ $(k<|{\vec\nu}|)$ from vector ${\vec\nu}$, then the resulting vector will be denoted by $\nu_{i_1,i_2,\ldots,i_k}$.

If we add components with numbers $j_1,j_2,\ldots,j_k$ to vector ${\vec\nu}$, then the resulting vector will be denoted by $\nu^{j_1,j_2,\ldots,j_k}$.

Denote the set of all possible vector $\nu_{i_1,i_2,\ldots,i_k}$ by $S^\downarrow({\vec\nu})$, and the set of all possible vector $\nu^{j_1,j_2,\ldots,j_k}$ by $S^\uparrow({\vec\nu})$.
Denote by $S(\vec\nu)$ the set of all non-empty subsets from the set $\{\nu_1,\nu_2,\ldots,\nu_{|{\vec\nu}|}\}$, and $\widehat {S(\vec\nu)}$ -- the set of all non-empty finite subsets from the set $\big\{\mathbb Z\setminus \{\nu_1,\nu_2,\ldots,\nu_{|\nu|}\}\big\} $.

3. 
Let $X_t = (\vec\nu;\vec y)$, $\vec y = (y_{\nu_1},y_{\nu_2}, \ldots y_{\nu_{|\nu|}})$.
For all $M=\{m_1,m_2,\ldots,m_k\}\in S(\vec\nu)$, with probability $\lambda_{m_1,m_2,\ldots,m_k}(X_t) \Delta+o(\Delta)=\lambda_{M}(X_t) \Delta+o(\Delta)$ during time $\Delta>0$ a spontaneous transition $X_t$ to the state $(\vec\nu,\vec z)$ takes place, where 
$$
\begin{array}{ll}
    z_i=y_i+  \Delta, & \mbox{if}i\in S(\vec\nu(t))\setminus M;\\
    z_i=\theta_i  \Delta,\; \theta_i\in(0;1), & \mbox{if}i\in M.
\end{array}
$$
\\
4. 
Let $X_t = (\vec\nu;\vec y)$, $\vec y = (y_1, \ldots y_{|{\vec\nu}|})$.
For all $M=\{m_1,m_2,\ldots,m_k\}\in \widehat{S(\vec\nu)}$, with probability $\mu^{m_1,m_2,\ldots,m_k} (X_t) \Delta+o(\Delta) =\mu^M (X_t) \Delta+o(\Delta)$ during time $\Delta>0$ a spontaneous transition $X_t$ to the state $(\vec\nu^{\,j_1,j_2,\ldots,j_k},\vec z)$ takes place, where  
$$
\begin{array}{ll}
    z_i=y_i+  \Delta, & \mbox{if}i\in S(\vec\nu);\\
    z_i=\theta_i  \Delta,\; \theta_i\in(0;1), & \mbox{if}i\in  M. 
\end{array}
$$
5. 
Let $X_t = (\vec\nu;\vec y)$, $\vec y = (y_1, \ldots y_{|{\vec\nu}|})$.
For all $M=\{m_1,m_2,\ldots,m_k\}\in S(\vec\nu)$, with probability $\lambda^{m_1,m_2,\ldots,m_k} (X_t) \Delta+o(\Delta)=\lambda^M (X_t) \Delta+o(\Delta)$ during time $\Delta>0$ a spontaneous transition $X_t$ to the state $(\vec\nu_{m_1,m_2,\ldots,m_k},\vec z)$ takes place, where 
$$
\begin{array}{ll}
    z_i=z_i+  \Delta, & \mbox{for all}i\in S^\downarrow({\vec\nu}). 
\end{array}
$$

6. 
Let $X_t = (\vec\nu,\vec y)$, $\vec y = (y_1, \ldots y_{|\nu|})$. 
With probability 
$$\Lambda(X_t)\bd1-\left[\suml_{M\in S(\vec\nu(X_t))}\big(\lambda_M (X_t) + \lambda^M (X_t)\big)\Delta + \suml_{M\in \widehat {S(\vec\nu(X_t))}}\mu^M (X_t)\right] +o(\Delta)$$ during time $\Delta>0$ a spontaneous transition $X_t$ to the state $(\vec\nu,\vec z)$ takes place, where 
$$
\begin{array}{ll}
    z_i=y_i+  \Delta, & \mbox{for all}i\in S^\downarrow({\vec\nu}). 
\end{array}
$$
\TR

\end{df}
\begin{rem}
It is easy to see that the Definition \ref{defMar} describes a Markov process that has a variable number of piecewise continuous components.

Transition probabilities are determined by functions of $X_t$, i.e. they are random but dependent on the state of the process $X_t$ at the time $t$.

Here $\lambda_{m_1,m_2,\ldots,m_k}(X_t)$ corresponds to the zeroing of several piecewise components of the process $X_t$, $\mu^{m_1,m_2,\ldots,m_k} (X_t)$ corresponds to adding several piecewise components to the process, and $\lambda^{m_1,m_2,\ldots,m_k}(X_t)$ it corresponds to the disappearance of several piecewise components from the process.

Usually, in queuing problems, the situation is used when $|M|=1$, and piecewise continuous components represent the intervals between arrivals of claims and their service times.

Situation $|M|>1$ may correspond to the batch arrival of requests or their batch service, etc.

The case when  $\Lambda(X_t)=\Lambda(\vec\nu(t))$ corresponds to the model described in Proposition \ref{defLin}. 
If we add condition  $\Lambda(X_t)=\Lambda(\vec\nu(t))$ to condition $\nu(X_t)|\equiv 1$, we get the process definition from \cite{Belyaev62,Koroluk75}.
\TR
\end{rem}

\subsection{Examples of piecewise linear Markov processes (PLMP)}

\begin{exmp}\label{ExRT}
Consider a mathematical model of the restorable element with a warm reserve.
We will describe the behaviour of this model using the process $X_t=((i_1(t),i_2(t));(x_1(t),x_2(t)))$, where $(i_1(t),i_2(t))=\vec\nu(t)$ indicates the state (working or being repaired) of the elements of the system, and the vector $\vec x=(x_1(t),x_2(t))$ denotes the elapsed time of both elements in the state indicated by $\vec\nu(t)$, i.e. for $k=1,2$:
\begin{figure}[h]
\centering
\begin{picture}(400,120)
\put(-10,35){\scriptsize $\hat t_0=0$}
\put(0,30){\vector(1,0){400}}
\put(0,100){\vector(1,0){400}}
\put(0,30){\circle*{3}}
\put(0,100){\circle*{3}}
\put(0,15){ELEMENT II.}
\thicklines
\qbezier(0,30)(50,40)(100,30)
\put(30,40){\scriptsize work [1]}
\put(100,35){\scriptsize $\hat t_1$}
\qbezier(100,30)(120,20)(140,30)
\put(100,15){\scriptsize repair [0]}
\put(140,35){\scriptsize $\hat t_2$}
\qbezier(140,30)(170,40)(200,30)
\put(147,40){\scriptsize work [1]}
\put(200,35){\scriptsize $\hat t_3$}
\qbezier(200,30)(230,20)(260,30)
\put(210,15){\scriptsize repair [0]}
\put(260,35){\scriptsize $\hat t_4$}
\qbezier(260,30)(297,40)(330,30)
\put(282,40){\scriptsize work [1]}
\put(320,35){\scriptsize $\hat t_5$}
\qbezier(330,30)(345,20)(360,30)
\put(345,15){\scriptsize repair [0]}
\put(355,35){\scriptsize $\hat t_6$}
\multiput(365,35)(10,0){5}{\circle*{2}}
\put(-10,105){\scriptsize $t_0=0$}
\put(0,85){ELEMENT I.}
\qbezier(0,100)(80,110)(160,100)
\put(30,110){\scriptsize work [1]}
\put(150,105){\scriptsize $ t_1$}
\qbezier(160,100)(200,90)(240,100)
\put(185,85){\scriptsize repair [0]}
\put(240,105){\scriptsize $ t_2$}
\qbezier(240,100)(280,110)(320,100)
\put(245,110){\scriptsize work [1]}
\put(320,105){\scriptsize $ t_3$}
\qbezier(320,100)(350,90)(380,100)
\put(343,85){\scriptsize repair [0]}
\put(375,105){\scriptsize $ t_4$}
\multiput(385,105)(10,0){3}{\circle*{2}}
\multiput(180, 15)(0,5){20}{\bb \line(0,1){3}}
\put(182,20){\scriptsize $\theta_1$}
\put(182,105){\scriptsize $\theta_1$}
\put(160,100){\gr\line (0,1){15}}
\put(180,110){\gr \vector(-1,0){20}}
\put(162,117){\scriptsize $x_1(\theta_1)$}
\put(140,30){\gr\line (0,-1){15}}
\put(180,15){\gr \vector(-1,0){40}}
\put(145,5){\scriptsize $y_1(\theta_1)$}
\multiput(280, 15)(0,5){19}{\bb \line(0,1){3}}
\put(240,100){\gr\line (0,-1){15}}
\put(282,20){\scriptsize $\theta_2$}
\put(282,92){\scriptsize $\theta_2$}
\put(280,85){\gr \vector(-1,0){40}}
\put(245,90){\scriptsize $x_2(\theta_2)$}
\put(260,30){\gr\line (0,-1){15}}
\put(280,15){\gr \vector(-1,0){20}}
\put(260,5){\scriptsize $y_2(\theta_2)$}
\multiput(340, 15)(0,5){21}{\bb \line(0,1){3}}
\put(342,32){\scriptsize $\theta_3$}
\put(342,102){\scriptsize $\theta_3$}
\put(320,100){\gr\line (0,1){15}}
\put(340,115){\gr \vector(-1,0){20}}
\put(320,120){\scriptsize $x_3(\theta_3)$}
\put(330,30){\gr\line (0,1){15}}
\put(340,45){\gr \vector(-1,0){10}}
\put(320,50){\scriptsize $y_3(\theta_3)$}
\end{picture}
\caption{This scheme is a visualization of the work and repair of a reliability system with a warm reserve. Here at the time $\theta_1$ we have $X_{\theta_1}=((1,0);(x_1(\theta_1),y_1(\theta_1)))$, the element II works; at the time $\theta_2$ we have $X_{\theta_2}=((0,0);(x_2(\theta_2),y_2(\theta_2)))$, both elements are in working state; at the time $\theta_2$ we have  $X_{\theta_3}=(1,0);(x_3(\theta_2),y_3(\theta_2)))$, and the system in failure state.}
\end{figure}
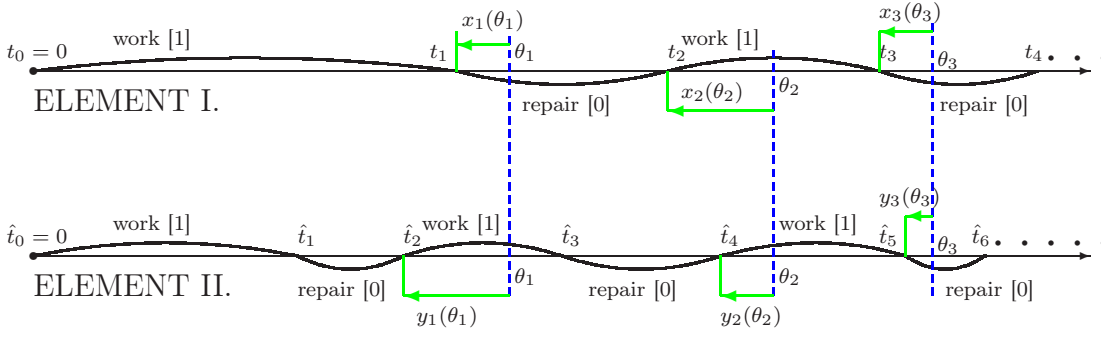\label{FigRT}
$$
\begin{array}{l}
 i_k(t) \bd 
\begin{cases}
0, & \mbox{ if if $i_k$-th element is in working state at the time} t;\\
1, & \mbox{ if it is repaired at the time} t;  
\end{cases}
\\
x_{i_k}(t)=t-\sup\{s:\;i_k(s)\neq i_k(t\}.
\end{array}
$$

The behaviour of both elements is described by the intensities of the failures and the repairs which depend on the full state of this reliability system $X_t$: 
$$
\begin{array}{l}
   \PP\left\{
   \begin{array}{l}
      i_1(t+\Delta)=|i_1(t)-1|,\, x_1(t+\Delta)\in(0;\Delta),\\
    i_2(t+\Delta)=i_2(t),\, x_2(t+\Delta)=x_2(t)+\Delta 
   \end{array}
   \Bigg|X_t\right\}=
   \begin{cases}
       \lambda_1(X_t)\Delta+o(\Delta), & \mbox{  if} i_1(t)=0;\\
       \mu_1(X_t)\Delta+o(\Delta), & \mbox{  if} i_1(t)=1; 
   \end{cases}
   \\\\
   \PP\left\{
   \begin{array}{l}
      i_2(t+\Delta)=|i_2(t)-1|,\, x_2(t+\Delta)=\in(0;\Delta),\\
    i_1(t+\Delta)=i_1(t),\, x_1(t+\Delta)=x_1(t)+\Delta 
   \end{array}
   \Bigg|X_t\right\}=
   \begin{cases}
       \lambda_2(X_t)\Delta+o(\Delta), & \mbox{  if} i_2(t)=0;\\
       \mu_2(X_t)\Delta+o(\Delta), & \mbox{  if} i_2(t)=1; 
   \end{cases}
   \\ \\
   \PP\left\{
   \begin{array}{l}
    i_1(t+\Delta)=i_1(t)-1,\, x_1(t+\Delta)=x_1(t),\\
    i_2(t+\Delta)=i_2(t),\, x_2(t+\Delta)=x_2(t)+\Delta 
   \end{array}
   \Bigg|X_t\right\}=
   \\
  \hspace{5cm} =1-(\lambda_1(X_t)+ \lambda_2(X_t)+ \mu_1(X_t)+\mu_2(X_t))\Delta + o(\Delta)
\end{array}
$$
for $\Delta\ll 1$.

This is a description of a Markov process on the state space $\XX\bd\{0,1\}^2\times \mathbb R_{\geqslant 0}^2$, i.e. the behaviour of the presented reliability model is described by PLMP.
In Fig. \ref{FigRT} a visualization of some realization of such a process is presented.
\TR
\end{exmp}

\begin{exmp}\label{ExQS}
Let the queuing system (QS) have an infinite number of servers and one incoming flow of ``events'' or customers ``of the same type'' such that the incoming flow is determined by a variable intensity, which depends on the full state of the QS.
Here the full state of QS at the time $t$ is: the total number $n(t)$ of customers in the QS, the elapsed times $x_i(t)$ $(i>0)$ of their service in the system, as well as $x_0(t)$ -- the time  elapsed since the last arrival of new customers to the QS.
So, the full state of this QS is the vector $\widehat X_t=(\vec\nu(t);\vec x(t))$. 
Here $\vec\nu(t)=(0, i_1(t), i_2(t),\ldots, i_{n(t)}(t)$, where $n(t)$ is the number of customers in the system at the time $t$, $0$ corresponds to the coordinate of the incoming flow, $ i_1(t),i_2(t), \ldots,i_{n(t)}(t)$ are the numbers of servers serving the corresponding customers.

We assume that all servers are identical, i.e. their numbers don't matter.
We will assume that the server numbers correspond to the order of arrival of customers in the QS.

Thus, the full state of this QS can be described by simplified vector $X_t=(n(t);\vec x(t))$.

\begin{figure}[h]
\centering
\begin{picture}(360,55)
\put(0,10){\vector(1,0){390}}
\put(0,10){\vector(0,1){50}}
\put(5,45){$\tiny n(t)$}
\thicklines
\put(0,0){$\tiny t_0$}
\put(0,10){\line(1,0){40}}
\multiput(40,10)(0,5){2}{\line(0,1){3}}
\put(40,0){$\tiny t_1$}
\put(40,20){\circle*{3}}
\put(40,10){\circle{4}}
\put(40,20){\line(1,0){20}}
\multiput(60,10)(0,5){4}{\line(0,1){3}}
\put(60,0){$\tiny t_2$}
\put(60,30){\line(1,0){40}}
\multiput(100,10)(0,5){4}{\line(0,1){3}}
\put(100,0){$\tiny t_3$}
\put(100,20){\line(1,0){15}}
\multiput(115,10)(0,5){4}{\line(0,1){3}}
\put(115,0){$\tiny t_4$}
\put(115,30){\line(1,0){25}}
\multiput(140,10)(0,5){4}{\line(0,1){3}}
\put(140,0){$\tiny t_5$}
\put(140,20){\line(1,0){10}}
\multiput(150,10)(0,5){2}{\line(0,1){3}}
\put(150,0){$\tiny t_6$}
\put(150,10){\line(1,0){20}}
\multiput(170,10)(0,5){2}{\line(0,1){3}}
\put(170,0){$\tiny t_7$}
\put(170,20){\circle*{3}}
\put(170,10){\circle{5}}
\put(170,20){\line(1,0){30}}
\multiput(200,10)(0,5){4}{\line(0,1){3}}
\put(198,0){$\tiny t_8$}
\put(200,30){\line(1,0){15}}
\multiput(215,10)(0,5){4}{\line(0,1){3}}
\put(213,0){$\tiny t_9$}
\put(215,20){\line(1,0){10}}
\multiput(225,10)(0,5){2}{\line(0,1){3}}
\put(223,0){$\tiny t_{10}$}
\put(225,10){\line(1,0){15}}
\multiput(240,10)(0,5){2}{\line(0,1){3}}
\put(240,0){$\tiny t_{11}$}
\put(240,20){\circle*{3}}
\put(240,10){\circle{5}}
\put(240,20){\line(1,0){25}}
\multiput(265,10)(0,5){4}{\line(0,1){3}}
\put(265,0){$\tiny t_{12}$}
\put(265,30){\line(1,0){15}}
\multiput(280,10)(0,5){6}{\line(0,1){3}}
\put(280,0){$\tiny t_{13}$}
\put(280,40){\line(1,0){20}}
\multiput(300,10)(0,5){8}{\line(0,1){3}}
\put(300,0){$\tiny t_{14}$}
\put(300,50){\line(1,0){30}}
\multiput(330,10)(0,5){8}{\line(0,1){3}}
\put(330,0){$\tiny t_{15}$}
\put(330,40){\line(1,0){10}}
\multiput(340,10)(0,5){6}{\line(0,1){3}}
\multiput(355,30)(10,0){3}{\circle*{2}}
\end{picture}
\caption{An example of a graph for the QS described in the Example \ref{ExQS}.}
\label{FigQS}
\end{figure}
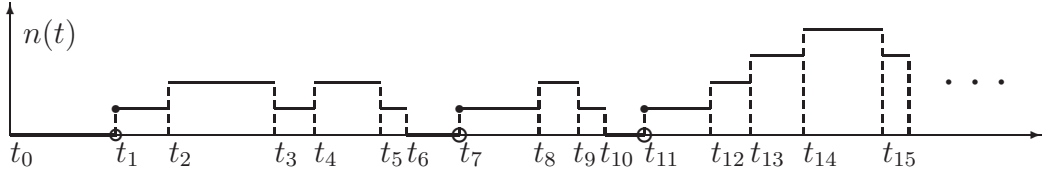

We assume that the rate of service depends on the full state $X_t$ of the QS, as well as the intensity of the incoming flow. 
Then, simply speaking, we assume
$$
\begin{array}{l}
  P(n(t-\Delta)=n(t)+1 \mid X_t) = \lambda(X_t) \Delta + o(\Delta);\\
  P(n(t-\Delta)=n(t)\pm 2 \mid X_t) = o(\Delta);\\
  P(X_{t+\Delta}=(n(t)-1;x_0(t)+\Delta,x_1(t)+\Delta,\ldots,x_{i-1}(t)+\Delta,x_{i+1}(t)+\Delta,\ldots\\
  \hspace{8cm}\ldots,x_{n(t)}+\Delta\mid X_t)  =\mu_i(X_t) \Delta + o(\Delta),
\end{array}
$$
where $t\ll1$.

It is easy to see that the process $X_t$ is Markov on the state space $\bigcup\limits_{k=1}^\infty \left[\mathbb N\times \mathbb R^k_{\geqslant 0}\right]$ as well as the process $\widehat X_t$ on the space $\bigcup\limits_{k=1}^\infty \left[\mathbb N^k\times \mathbb R^k_{\geqslant 0}\right]$.

Both of these processes are PLMP.
Such systems with different conditions for intensities $\lambda(X_t)$ and $\mu_i(X_t)$ have been considered in \cite{Veretennikov13QS,Veretennikov17QS,VeretennikovZverkina14}, et al.
\TR
\end{exmp}

\begin{rem}
On Fig. \ref{FigQS} the times $t_1$, $t_7$, $t_{11}$ are marked when all piecewise-continuous components of the process are set to zero: $X_{t_1}=X_{t_7}=X_{t_{11}}=(1;0,0)$. 
As the process $X_t$ is Markov, these times are the regeneration times.
\TR
\end{rem}

In Fig. \ref{FigRT}, as well as in Fig. \ref{Fig2x}, it can be seen that the presented PLMP does not have regeneration points, while in Fig. \ref{FigQS} we marked regeneration times.
Processes with regeneration times are more convenient for research.
To study their asymptotic behaviour, one can use their embedded renewal processes.
So, the PLMP's are very important in queuing theory and related fields, and in many situations, they are {\it regenerative} processes.

\subsection{Regenerative Process}
\begin{df}
[Regenerative Process (see., e.g., \cite{AfanasevaTkachenko19})]\label{defReg}
 The process $(X_t,\,t\geqslant 0)$ on a probability space $(\Omega,\mathscr{F} , \PP )$, with a measurable state space $(\mathscr{X}  , \mathscr{B} (\mathscr{X} ))$ is regenerative,  if there exists an increasing sequence $\left\{ \theta_n\right\} $  $(n\in\mathbb{Z}_+)$ of Markov moments with respect to the filtration $\mathscr{F} _{t\geqslant  0}$ such that the sequence
$$
\left\{ \Theta_n\right\} =\left\{ X_{t+\theta_{n-1}}-X_{\theta_{n-1}}, \theta_n- \theta_{n-1},t\in[\theta_{n-1},\theta_n)\right\} ,\quad n\in\mathbb{N}
$$
consists of independent identically distributed (i.i.d.) random elements on \linebreak $(\Omega,\mathscr{F} , \PP )$. If $\theta_0\neq 0$, then the process $(X_t,\,t\geqslant 0)$ is called delayed.

The periods $[\theta_{n-1},\theta_n]$ are named {\it regeneration periods}.
\TR
\end{df}

The times $\theta_i$ are called {\it regeneration times}, and they form the renewal process with the renewal times $\xi_i\bd\theta_i-\theta_{i-1}$.
Denote this embedded renewal process by $N_t$, and its backward renewal time -- by $B_t$ (see Definition \ref{defRen}).

Let $F(s)=\PP \left\{ \xi_n\leqslant  s\right\} $ be the distribution function of regeneration period; {\it everywhere in the future we suppose that the distribution $F$ is non-lattice.}

If the process $X_t$ is regenerative, then at any time $t$ on some regeneration period  $[\theta_{n-1},\theta_n]$ ($t\in [\theta_{n-1},\theta_n]$) the distribution of $X_t$ is the function of $t-\theta_{n-1}=B_t$, where $B_t$ is the time elapsed since the start of the current regeneration period.
These times $\theta_i$ form an embedded renewal $N_t$ process with the backward renewal time $B_t$ at the time $t$.

The behaviour of the process $B_t$ is defined by the distribution function (d.f.) $F(\cdot)$ of the renewal period $\xi_i$ of the embedded renewal process $N_t$.
Also, the behaviour of the process $B_t$ is determined by the intensity function $\lambda(t)$ of the d.f. of $\xi_i$ (see Definitions \ref{lambda} and \ref{INT}).
 
Further, we note that the process $B_t$ is ergodic if $\EE\,\xi_i<\infty$, i.e. the distribution $\PPP^B_t$ of the process weak converges to the stationary invariant distribution ($\PPP^B_t\Longrightarrow \PPP^B$), then the process $X_t$ is also ergodic, and it's distribution $\PPP_t$ weak converges to the stationary invariant distribution $\PPP$: $\PPP_t\Longrightarrow \PPP$ -- see, e.g., \cite{Borovkov76}.

Moreover, if $\|\PPP^B_t- \PPP^B\|_{TV}\leqslant \psi(t)$, then for the distribution $\PPP_t^X$ of the process $X_t$:   $\PPP^X_t\Longrightarrow \PPP^X$, and $\|\PPP^X_t- \PPP^X\|_{TV}\leqslant \psi(t)$, where $\|\PPP_t- \PPP\|_{TV}$ is a distance in the total variation metrics: $\|\PPP_t- \PPP\|_{TV}\bd \supl_{A\in \BBB(\XX)}|\PPP_t(A)-\PPP(A)|$.

As well as known, in the case when the renewal time of an embedded renewal process $\xi$ has finite moments $\EE\,\xi^k$, $k>1$, then for all $\ell\leqslant k-1$ there exists some constant $C(\ell)$ such that 
$\|\PPP^B_t- \PPP^B\|_{TV}\leqslant\dd\frac{C(\ell)}{t^\ell}$ (see, e.g., \cite{Borovkov76,Thorisson00}) and, accordingly, $\|\PPP^X_t- \PPP^X\|_{TV}\leqslant\dd\frac{C(\ell)}{t^\ell}$.

But the method of the proof of this result in \cite{Borovkov76, Thorisson00} can not give the bounds for the constant $C(\ell)$.

Above, we will demonstrate the method of the calculation of an upper rough bound for these constants based on the coupling method and generalization of Lorden's inequality.

In \cite{Veretennikov13QS, Veretennikov14QS, Veretennikov16RT, Veretennikov19RT, VeretennikovZverkina14, VeretennikovZverkina16}, the coupling method was used for the analysis of the recurrence properties of some queuing and reliability systems. 
But in this paper, there is no algorithm for the calculation of an upper bound for the constant $C(\ell)$.  

\section{Bounds for convergence rate of the  distribution of $B_t$ in the simplest case}\label{S3}
Here we give a short version of the construction of an upper bound for the convergence rate of the distribution of an embedded renewal process backward time given in the papers \cite{ZverkinaDCCN16} and \cite{ZverkinaDCCN17} -- in the total variation metrics.

Firstly, recall some useful facts.

\subsection{Basic Coupling Inequality}
\begin{prop}[Basic Coupling Inequality]
If two homogeneous Markov processes (process defined on some probability space $(\Omega,\FFF,\PP)$) with the same transition function but with different initial states coincide at the time $\tau$, then after the time $\tau$ their distributions are equal.
Thus,
\begin{multline*}
|\PP\{X_t\in S\}-\PP\{X_t'\in S\}|=
\\
=|\PP\{X_t\in S\,\&\,\tau> t\}-\PP\{X_t'\in S\,\&\,\tau> t\}|+ |\PP\{X_t\in S\,\&\,\tau\leqslant t\}-\PP\{X_t'\in S\,\&\,\tau\leqslant t\}|\leqslant
\\
\leqslant \PP\{\tau> t\}=\PP\{\varphi(\tau)> \varphi(t)\}\leqslant \dd\frac{\EE\,\varphi(\tau)}{\varphi(t)},
\end{multline*}
where $\varphi(t)$ is an increasing positive function.

Thus, 
$$
\left\|\PPP^{X}_t-\PPP^{X'}_t \right\|_{TV}\bd \supl_{S\in \FFF_t}|\PP\{X_t\in S\}-\PP\{X_t'\in S\}|\leqslant \dd\frac{\EE\,\varphi(\tau)}{\varphi(t)},
$$
where $\PPP^{X}_t$ and $\PPP^{X'}_t$ are the distributions of the processes $X_t$ and $X_t'$ accordingly.
\TR
\end{prop}
Further, we will use the function $\varphi(s)\bd s^\ell$, $\ell>0$.

The Basic coupling inequality firstly was used for the Markov chains in discrete time.
But now we are interested in the use of the coupling method for the Markov processes in continuous time, -- namely, for renewal processes with arbitrary d.f. of renewal times.

Therefore we will use the modification of a coupling method adapted for Markov processes in continuous time. 
This modification is ``successful coupling'' invented in \cite{Griffeath75}, see also \cite{Lindvall}.
\subsection{Successful coupling}\label{SC}
\begin{df}[Successful coupling]
Let the pair $\Bigg(\left (X_t^{(1)},X_t^{(2)} \right),t\geqslant0 \Bigg)$ is the pair of Markov processes with identical transition probabilities, but their initial state can be different: $X_0^{(1)}\neq  X_0^{(2)}$. 
Consider the  paired stochastic process $\mathscr{Z}_t= \left(\left (Z_t^{(1)},Z_t^{(2)} \right),t\geqslant0 \right)$ constructed (in a special probability space) the such that:

1. For all $t\geqslant 0$ the random variables $X_t^{(i)}$ and $Z_t^{(i)}$ have the same distribution, $i=1,2$, i.e. $X_t^{(i)}\bD Z_t^{(i)}$.

2. $\mathbf{E}\,\tau  \left(Z_0^{(1)},Z_0^{(2)} \right)<\infty$,  where $$\tau  \left(Z_0^{(1)},Z_0^{(2)} \right)={\tau}  (\mathscr Z_0 )\stackrel{\rm{df}}{=\!\!\!=\!\!\!=} \inf \left\{ t\leqslant   0:\,Z_t^{(1)}=Z_t^{(2)} \right\} .
$$

3.  $Z_t^{(1)}=Z_t^{(2)}$ for all $t\geqslant {\tau}  \left(Z_0^{(1)},Z_0^{(2)} \right)$.

The paired stochastic process $\mathscr{Z}_t=\left(\left (Z_t^{(1)},Z_t^{(2)} \right),t\geqslant0 \right)$ satisfying the conditions 1--3 is called {\it successful coupling}.
\TR
\end{df}
\begin{rem}
The finite-dimensional distributions of the process $\left(Z_t^{(i)},\,t\geqslant0\right)$  may differ from the finite-dimensional distributions of $\left(X_t^{(i)},\,t\geqslant0\right)$; the processes $Z_t^{(1)}$ and $Z_t^{(2)}$ can be dependent.
\TR
\end{rem}
For all $A\in \mathscr{B}(\mathscr{X})$ ($\XX$ is a state space of the processes $X_t^{(i)}$) we can use the coupling inequality in the form for the distributions $\PPP^{X_0^{(i)}}_t$ of the processes $X_t^{(i)}$ with initial states  $X_0^{(i)}$:
\begin{multline*}
\left|\PPP^{X_0^{(1)}}_t(A) - \PPP_t^{X_0^{(2)}}(A) \right|= \left|\PP  \left\{ X_t^{(1)}\in A \right\} - \PP  \left\{ X_t^{(2)}\in A \right\}  \right|=
  \\
= \left|\PP  \left\{ Z_t^{(1)}\in A \right\} - \PP  \left\{ Z_t^{(2)}\in A \right\}  \right| \leqslant
  \PP  \left\{  {\tau}\left  (Z_0^{(1)},Z_0^{(2)}\right )\geqslant t \right\}
  \leqslant   \\
 \leqslant\frac{\mathbf {E} \,\varphi \left({\tau}  \left(Z_0^{(1)},Z_0^{(2)} \right) \right)}{\varphi(t)}\leqslant   \frac{C \left(Z_0^{(1)},Z_0^{(2)} \right)}{\varphi(t)}
\end{multline*}
for some constant $C \left(Z_0^{(1)},Z_0^{(2)} \right)$.

Thus, $ \dd \left\|\PPP^{X_0^{(1)}}_t-\PPP^{X_0^{(2)}}_t \right\|_{TV}\leqslant \dd\frac{C \left(Z_0^{(1)},Z_0^{(2)} \right)}{\varphi(t)}. $

As $Z_0^{(i)}=X_0^{(i)}$, the right-hand side of the inequality  depends only on $X_0^{(i)}$.
Then, if $\PPP_t^{X_0^{(1)}}\Rightarrow \PPP $ for all initial states, we can use the integration of  this inequality  with respect to the measure $\PPP $, i.e.
\begin{equation}\label{int}
\dd \left\|\PPP^{X_0^{(1)}}_t-\PPP \right\|_{TV}\leqslant \dd\intl_{x\in\XX}\frac{C \left(Z_0^{(1)},x \right)}{\varphi(t)} \PPP(\ud x).    
\end{equation}

\subsection{Basic Coupling Lemma}
\begin{lem}[Basic Coupling Lemma (\cite{VeretennikovButkovsky13,Kato14} et al.)]\label{BCL}
Let $f_i(s)$ be the distribution density of r.v. $\theta_i$ ($i=1,2$).
And let $$\intl_{-\infty}^ \infty \min(f_1(s),f_2(s))\ud s=\varkappa>0.$$

Then on some probability space there exists two random variables $\vartheta_i$ such that $\vartheta_i\bD \theta_i$, and $\PP\{\vartheta_1=\vartheta_2\}\geqslant \varkappa$. 
\TR
\end{lem}
\begin{df}
The value $\dd\intl_{-\infty}^ \infty \min(f_1(s),f_2(s))\ud s=\varkappa$ is called {\it common part} of the distributions of $\theta_i$.
\TR
\end{df}
\begin{rem}
The Lemma \ref{BCL} can be extended for many random variables.
\end{rem}
\begin{lem}[Useful generalization of the Basic Coupling Lemma (\cite{ZverkinaCN19})]\label{GBCL}
Let $f_i(s)$ be the distribution density of r.v. $\theta_i$ ($i=1,\ldots,n$).
And let $$\dd\intl_{-\infty}^ \infty \min\limits_{i=1,\ldots,n}(f_i(s))\ud s=\varkappa>0.$$
Then on some probability space there exists $n$ random variables $\vartheta_i$ ($i=1,\ldots,n$) such that $\vartheta_i\bD \theta_i$ for all $i=1,\ldots,n$, and $\PP\{\vartheta_1=\vartheta_2=\ldots=\vartheta_n\}\geqslant \varkappa$.
\TR
\end{lem}

\subsection{Lorden's inequality}

\begin{thm}[Lorden's inequality \cite{Lorden70}]\label{LI}
Lorden's inequality states that the expectation of the backward renewal time is bounded as
\begin{equation}
\EE\, B _t
 \leqslant  \frac{\EE\, \xi ^2}{\EE\, \xi}\bd \Xi\big(=\mbox{functional of}F(s)\big).
\end{equation}
\TR
\end{thm}

\subsection{Idea of construction of successful coupling for two backward renewal times processes}\label{constr}

Consider two renewal processes $N_t^{(1)}$  and $N_t^{(2)}$  which start at different times before the time $t=0$.
So, at the time $t=0$, their backward renewal times can be different.
Consider the Markov processes $B_t^{(1)}$ and $B_t^{(2)}$ -- the backward renewal times of considered renewal process; the values $b_1\bd B_0^{(1)}$ and \linebreak $b_2\bd B_0^{(2)}$ can be different, but the transition probabilities of both these processes are defined by the (generalized) intensity $\lambda(s)$. 
Here we suppose that $\lambda(s)$ is a.s. positive, and, accordingly, the d.f. $f(s)=F'(s)$ of the renewal times is also a.s. positive. 

{\it Our goal is to create on some specific probability space a pair of processes $\left (Z_t^{(1)},Z_t^{(2)} \right)$ in such a way that at each time $t\geqslant 0$ the distributions of the original and created pair of processes coincide $\left[\left (Z_t^{(1)},Z_t^{(2)} \right)\bD \left (B_t^{(1)},B_t^{(2)} \right)\right]$, and the processes $Z_t^{(1)}$, $Z_t^{(2)}$ will coincide at some time $\tau$, such that $\EE\,\tau<\infty$.}

So, for the pair $\BBB_t\bd\left(\left (B_t^{(1)},B_t^{(2)} \right),t\geqslant0 \right)$ we will construct (in a special probability space) the paired stochastic process $\mathscr{Z}_t= \left(\left (Z_t^{(1)},Z_t^{(2)} \right),t\geqslant0 \right)$ such that it is a successful coupling for $\BBB_t$ in the sense given in Section \ref{SC}.

To do this, we will start with a copy of processes $\left (B_t^{(1)},B_t^{(2)} \right)$, and then, at Markov times, we will change the pair (probability space + process $\ZZZ_t$) in such a way that for each such change, coupling of $Z_t^{(1)}$ and $Z_t^{(2)}$ can occur with some positive probability.

For this aim, let us generate $\UU_k$, $\UU_k'$,  $\UU_k''$,  $\UU_k'''$ ($k\in \mathbb Z_{\ge 0}$), which are i.i.d. r.v.'s  uniformly distributed on $[0,1)$.
These r.v.'s will be used in the next constructions.

At the time $t=0$, the distributions of the forward renewal times of the processes $N_t^{(1)}$  and $N_t^{(2)}$ are defined by the values $b_1$ and $b_2$:
\begin{equation}\label{W}
F_{b_i}(s)\bd \PP\left\{W_0^{(i)}\leqslant  s|B_0^{(i)}=b_i\right\}= \dd\frac{F(s+b_i)-F(b_i)}{1-F(b_i)}.    
\end{equation}
This is the distribution of the forward (residual) renewal time given the value of the backward renewal time.

For construct renewal process $(N^{b_1}_t,\,t\geqslant 0)$ with the initial state of backward renewal time $B_0^{(1)}=b_1$, it is enough calculate it's renewal times $t_0$, $t_1$, $t_2$, \ldots.
 
Put $t_0\bd F_{b_1}^{-1}(\UU_0)$, then $t_1\bd t_0+F^{-1}(\UU_1)$, \ldots, $t_j\bd t_{j-1}+F^{-1}(\UU_j)$, $j>1$, \ldots.

(Recall that $\PP\{F^{-1}\UU_j\leqslant s\}=F(s)$.)

\begin{figure}[h]
    \centering
   \begin{picture}(330,40)
\put(0,20){\vector(1,0){330}}
\thicklines
{\qbezier(0,33)(40,34)(60,20)}
{\qbezier(60,20)(80,40)(100,20)}
\qbezier(100,20)(140,40)(180,20)
\qbezier(180,20)(200,40)(220,20)
\qbezier(220,20)(250,40)(280,20)
\qbezier(280,20)(290,40)(300,20)
\qbezier(300,20)(320,31)(330,30)
\thinlines
\put(15,35){\scriptsize$\xi_0\sim  F_{b_1}$}
\put(67,35){\scriptsize$\xi_1\sim  F$}
\put(125,35){\scriptsize$\xi_2\sim  F$}
\put(185,35){\scriptsize$\xi_2\sim  F$}
\put(235,35){\scriptsize$\xi_4\sim  F$}
\put(210,13){\scriptsize$t_3$}
\put(58,13){\scriptsize$t_0$}
\put(98,13){\scriptsize$t_1$}
\put(178,13){\scriptsize$t_2$}
\put(278,13){\scriptsize$t_4$}
\put(298,13){\scriptsize$t_5$}
\put(0,13){\scriptsize$0$}
\end{picture}
    \caption{Construction of the process $Z_t^{(1)}$ with the same distributions as the process $N_t^{(1)}$.}
    \label{fig:Z1}
\end{figure}
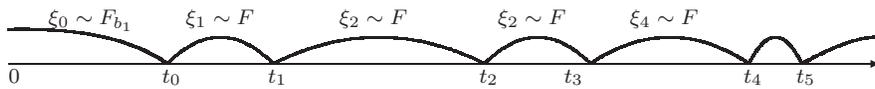

Analogically, for $B_t^{(2)}$ we calculate: $t'_0\bd F_{b_1}^{-1}(\UU'_0)$, $t'_1\bd t'_0+F^{-1}(\UU'_1)$, \ldots, $t'_j\bd t'_{j-1}+F^{-1}(\UU'_j)$, $j>1$, \ldots.

Now, for $\ZZZ_t$ we want to change the construction  of this pair of processes in such a way that at some time $\tau$ their renewals occur simultaneously. 

Our construction of the successful coupling is based on the next consideration.

$N_t^{(1)}$  (corresponding to $B_t^{(1)}$) and $N_t^{(2)}$ (corresponding to $B_t^{(2)}$) are the renewal processes.
Suppose that some renewal time of $N_t^{(1)}$ is $\hat t>\sup\{t:\; N_t^{(2)}<1\}$, that is, before the time $\hat t$ there was at least one renewal of the processes $N_t^{(2)}$.

Suppose that at the time  $\hat t$, the backward renewal time $B_{\hat t}^{(2)}$ of the process $N_t^{(2)}$ is equal to $a\geqslant 0$:  $B_{\hat t}^{(2)}=a$.

At this time, we can apply the Basic Coupling Lemma \ref{BCL} to the next renewal period of the process $N_t^{(1)}$ and the forward renewal time of the process $N_t^{(2)}$, which have the d.f. $\dd F_a(s)=\frac{F(s+a)-F(a)}{1-F(a)}$.

So, with probability $\varkappa(a)\bd\dd \intl_0^\infty \min(f(s), f_a(s))\ud s>0$, where  
$f(s)=F'(s)$, and $ f_a(s)=F'_a(s)$ the next renewals of the processes $N_t^{(1)}$ and $N_t^{(2)}$ coincide.
$\varkappa(a)>0$ because $f(s)>0$ a.s.; moreover, $\infl_{a\in[0;T]}\varkappa(a)>0$ for all $T>0$.

But the distribution of the random value $B_{\hat t}^{(2)}=a$ is unknown.

The Theorem \ref{LI} (Lorden's inequality) makes it possible to indicate the bounds for the value  $B_{\hat t}^{(2)}=a$ with known probability -- after the time $t_0'$ (i.e. before the first renewal of the process $N_t^{(2)}$ we cannot say about the use of Lorden's inequality).

Really, for all $\Theta>\Xi$ (see the Theorem \ref{LI}), the Markov inequality gives: $\PP\left\{B_t^{(2)}\leqslant \Theta\right\}\geqslant p_\Theta\bd 1-\dd \frac{\Xi}{\Theta}$.
Thus, at the all renewal times $\hat t$ of the process $N_t^{(1)}$ (after the first renewal of the process $N_t^{(2)}$), $B_{\hat t}^{(2)}=a\in [0, \Theta]$ with probability greater than $p_\Theta$.

So, in the case $B_{\hat t}^{(2)}=a\in [0, \Theta]$, at the time $\hat t$ we can apply the Basic Coupling Lemma \ref{BCL}, and with probability greater then  $\varkappa_\Theta\bd\infl_{a\in(0,\Theta)}\varkappa(a)$ the next renewal of both processes coincide.
$\varkappa_\Theta>0$ because we suppose that the density $f(s)=F'(s)$ exists and is positive almost sure for $s>0$.

Therefore, after any renewals of the process $N_t^{(1)}$ (after the first renewal of the process $N_t^{(2)}$), we can prolong the construction of the pair \linebreak $\mathscr{Z}_t= \left(\left (Z_t^{(1)},Z_t^{(2)} \right),t\geqslant0 \right)$ by such a way, that the processes $Z_t^{(1)}$ and $Z_t^{(2)}$ coincide at the next renewal.
This time of the coincidence of the processes $Z_t^{(1)}$ and $Z_t^{(2)}$ is a coupling epoch $\tau=\tau\left (Z_0^{(1)},Z_0^{(2)} \right)=\tau\left (B_0^{(1)},B_0^{(2)} \right)=\tau(b_1,b_2)$.

So, the coupling epoch of the pair $\mathscr{Z}_t= \left(\left (Z_t^{(1)},Z_t^{(2)} \right),t\geqslant0 \right)$ is a geometrical sum of conditional renewal times given the coincidence happened at the last renewal period. 

Thus, the coupling epoch $\tau(b_1,b_2)$ for successful coupling is
\begin{multline*}
 \tau(b_1,b_2)=   \left\{ T_0+\suml_{i=1}^n\xi_i^{(2)}\;\Bigg|\; \mbox{coupling epoch is}t_{n+1}'\right\}
 \\
 \mbox{ with probability  greater then}(1-p_0\varkappa_\Theta)^{n-1}\bd \pi^{n-1},
\end{multline*}
where $T_0\bd \max \left\{W_0^{(1)},W_0^{(2)}\right\}\le W_0^{(1)}+W_0^{(2)}$ (see (\ref{W})).
So, it can calculate an upper bounds for $\EE\,(\tau(b_1,b_2))^\ell$ if there exists finite $\EE\,\xi_i^k$, $k>\ell+1$.
  
This is a technical problem that can be solved with the use of Jensens's inequality and inequality   $\EE(\xi|A)\PP(A)\leqslant  \EE(\xi)$.
  
After obtain the bounds $\EE\,(\tau(b_1,b_2))^\ell\leqslant C(\ell,b_1,b_2)$ we have \linebreak $\|\PPP^{b_1}_t-\PPP^{b_2}_t\|_{TV}\leqslant \dd\frac{C(\ell,b_1,b_2)}{t^\ell}$, where $\PPP^{b_i}_t$ is a distribution of the process $B_t^{(i)}$ with initial state $B_0^{(i)}=b_i$ --  see, e.g., \cite{ZverkinaDCCN16,ZverkinaDCCN17}.

The stationary measure of $B_t$ in our conditions is well-known (see, e.g., \cite{Smith58}): $\PPP(0;x)=\dd\frac{1}{\EE\,\xi}\intl_0^x (1-F(s))\ud s$. 
Then no great difficulties of the integration $\dd\intl_0^\infty C(\ell,b_1,b_2)\ud \PPP(b_2)$ by stationary measure.

\begin{rem}
 In the first attempts to analyse the recurrence property of the piecewise linear processes in queuing theory, the special Lyapunov function was used.
 Thus, all distributions of all renewal times defining the queuing process had to be absolutely continuous (\cite{Veretennikov13QS, Veretennikov14QS}).
 
 But the use of Lorden's inequality made it possible to find an upper bound for the rate of convergence of the distribution of the queuing process (\cite{ZverkinaJMS21}).
 
 In all these papers, the input flow has an intensity that is separated from zero by a positive constant.
\end{rem}

\section{Quasi-renewal and quasi-regenerative processes}
Usually, in queuing theory and related fields, the ``classical'' renewal processes are used (\cite{Smith58}) or their combinations as alternating renewal processes (see, e.g., \cite{Sheldon10}).

Earlier it was said that the renewal process is a counting process. 

For the renewal process, all renewal periods have an identical distribution.

But in Example \ref{ExRT}, as a rule, the distribution of work and repair time changes at different time intervals, since the functioning of both elements depends on each other.
Therefore, in such a situation, the sequence of periods of work or repair of both elements is not a renewal process. 
However, for such a reliability model, with some restrictions for intensities, the ergodicity and rate of convergence of the distribution of the state of this system to the limit distribution was studied in \cite{ZverkinaCN19}.

In Example \ref{ExQS}, also under general assumptions about the intensities, the periods between the arrivals of customers are not the renewal times of some renewal processes, since their distribution varies depending on the full state of the QS.
Also, a sequence of customer service periods does not constitute a renewal process.

However, ergodicity was also established for such QS's and bounds were obtained for the rate of convergence of the distribution of the state of the system to its limiting value - naturally, under certain restrictions on the intensities. 
Probably, the first case of the analysis of such a system was given in \cite{Veretennikov13QS}. 
Now, the analysis of such systems is given in \cite{Veretennikov14QS, Veretennikov16RT, Veretennikov17QS, Veretennikov19RT, VeretennikovZverkina14, VeretennikovZverkina16, ZverkinaFPM20} et al.


Further, in some models described by piecewise linear processes, the studied processes are not regenerative, i.e. the duration of the periods between the hits to a certain fixed state are not i.i.d. r.v.'s. 



In addition, they can be ``weak'' dependent because their distributions depend on the behaviour of some other parameters of the studying system, -- and these parameters depend on all other distributions. 

Moreover, in many situations, the switch between different states of some element of the reliability model (or in queuing or network systems) can have some duration.
Thus, we also need to consider the delayed renewal process.

This is a reason for the attempt presented here to analyze the behaviour of such non-standard near-renewal processes.

\subsection{Quasi-renewal processes}
Consider the counting process $\widetilde N_t\bd \dd\suml _{i=1} ^\infty \1\left \{ \suml _{s=1} ^i\widetilde \xi _k\leqslant t\right \} $, where the random variables $\widetilde \xi_i$ can be dependent ant they can have different distributions.
In the case of arbitrary distributions of random variables $\widetilde \xi_i$ and the case of an arbitrary dependence between them, the study of the behaviour of the process  $\widetilde N_t$ is naturally impossible. 

\begin{df} [Quasi-regenerative process]\label{QR}
The process $\widetilde N_t\bd \dd\suml _{i=1} ^\infty \1\left \{ \suml _{s=1} ^i\widetilde \xi _k\leqslant t\right \} $ is named quasi-renewal process if the r.v.'s $\widetilde \xi_i$ are defined by their generalized intensity $\lambda_i(s)$ (see Definition \ref{INT}) which satisfy the following conditions
\\
1. The (generalized) measurable non-negative functions $ \varphi(s)$ and $Q(s)$ exist such that for all $s\geqslant 0$, ${ \varphi(s)\leqslant\lambda_j(s) \leqslant Q(s)}$;
\\
2. $\dd\intl_0^\infty  \varphi(s) \ud s = \infty$, and $\dd\intl_0^\infty x^{k-1}  \exp\left(-\intl_0^x  \varphi(s)\ud s\right)\ud x=M_k<\infty $ for some $k\geqslant 2$;
\\
3. $Q(s)$ is bounded in some neighbourhood of zero;
\\
4. There exists the constant $T\geqslant 0$ such that $ \varphi(s)>0$ a.s. for all $s>T$.
\TR
\end{df}
\begin{rem}
If $T>0$, the process $\widetilde N_t$ is the delayed process.
The time $T$ bounds the delay of the quasi-renewal times.
\TR
\end{rem}

\begin{rem}\label{r0}
The conditions (2) and (3) ensure that:
 $$\mathbb{E} \,\xi_i>0,\qquad \mathbb{E} \,(\xi_i)^2<\frac{M_2}{2}<\infty \qquad \mbox{Var}\,\xi_i^2>0.$$
\TR
\end{rem}
\begin{rem}
 Condition 4 is needed for the use of the coupling method. 
 \TR
\end{rem}
Note, that the one-dimensional Markov process cannot be quasi-regenerative; in the Examples \ref{ExRT}, \ref{ExQS} discussed above, quasi-regenerative processes are one of the components of the PLMP.

The use of quasi-renewal processes is based on the 
 \begin{thm}[Generalized Lorden's inequality (Kalimulina, GZ, 2019 -- \cite{KalimulinaZverkinaICSM2020})]\label{GLEthm}
 If conditions 1--3 are satisfied, then for the backward renewal time $B_t$ and forward renewal time $W_t$ of the quasi-renewal process the following inequality is true for $\ell\leqslant  k-1$
 \begin{equation}\label{GLE}
   \EE\,(B_t)^\ell\leqslant \EE\,\eta^\ell+\frac{\EE\,\eta^{\ell+1}}{(\ell+1)\EE\,\zeta}\bd \Xi_\ell,   
 \end{equation}
where 
$$\dd \PP\{\eta\leqslant  x\}= 1- \exp\left(\dd-\intl_0^x  \varphi(t)\ud t\right); \qquad
\PP\{\zeta\leqslant  x\}= 1- \exp\left(\dd-\intl_0^x  Q(t)\ud t\right).$$     
\TR
\end{thm} 
\begin{rem}
The same inequality is true for forward renewal time (under the same conditions):
 \begin{equation}\label{GLE1}
   \EE\,(W_t)^\ell\leqslant \EE\,\eta^\ell+\frac{\EE\,\eta^{\ell+1}}{(\ell+1)\EE\,\zeta}\bd \Xi_\ell,   
 \end{equation}
 \TR
\end{rem}
The quasi-regenerative process in reliability theory and queueing theory can be described by the multidimensional (or variable-dimensional) composition of quasi-renewal processes or alternating quasi-renewal processes.

For example, in \cite{Zverkina20smarty} the case when the intensity of input flow and service times of the queuing system satisfy the conditions 1-4 of the Definition \ref{QR}, and the input flow is only bounded.

Moreover, the application of Generalized Lorden's inequality and coupling method gave the bounds of convergence for the generalized Markov modulated process, i.e. for a process consisting of several delay quasi-renewal processes, the intensities of which depend on the backward renewal time of all components -- see \cite{ZverkinaDCCN20}.

In \cite{Veretennikov19RT} the reliability model from Example \ref{ExRT} was studied -- but in the condition that failure and repair intensities are bounded from zero by some positive function.

Since Section \ref{constr} suggests using Lorden's inequality \ref{LI} to construct successful coupling for a ``classic'' renewal processes, one might as well use Lorden's Generalized Inequality \ref{GLEthm} to construct successful quasi-recovery coupling of processes.

{\it The use of Generalized Lorden's Inequality gives the possibility to study non-regenerative processes and use the coupling method for calculating upper bounds for convergence rates in the case when the intensities are not separated from zero, and the switching between operating modes may be delayed.}

\section{Quasi-regenerative processes}
\begin{df}[Quasi-regenerative process]
 {\it The process $(X_t,\,t\geqslant 0)$ on a probability space $(\Omega,\mathscr{F} , \PP )$, with a measurable state space $(\mathscr{X}  , \mathscr{B} (\mathscr{X} ))$ is quasi-regenerative,  if on other probability space $(\widetilde \Omega,\widetilde{ \mathscr{F}} , \widetilde \PP )$ there exists regenerative Markov process $\widetilde X_t$ such that for all $t\geqslant  0$, the distribution $\PPP_t$ of $X_t$ is equal to the distribution $\widetilde\PPP_t$ of the process $\widetilde X_t$.
}
\end{df}   
\begin{rem}
This definition is similar to the previously given Definition \ref{SC} of "successful coupling".

Therefore, the construction of a ``regenerative copy'' for quasi-regenerative processes is based on the coupling method.  \TR
\end{rem}

To illustrate the concept of a quasi-regenerative process, we will show that under certain conditions for intensities $\lambda_i(X_t)$, $\mu_i(X_t)$, process $X_t$, which describes the reliability model presented in Example \ref{ExRT}, is quasi-regenerative.
\begin{thm}
Suppose, that
\\
{\bf a.} There exists (generalized) measurable non-negative functions $\varphi^{(j)}_i(t)$ and $\Phi^{(j)}_i(t)$  defined on the interval $(0;\infty)$, and the constants  $T_i^{(j)}$ such that for all $t>T_i^{(j)}$ the functions $\varphi^{(j)}_i(t)$are a.s. positive  $(i,j\in\{1;2\})$;
\\
{\bf b.} $\dd\intl_0^\infty \varphi^{(j)}_i(s)\ud s=\infty$, and  $\dd\intl_0^\infty x^{k-1}  \exp\left(-\intl_0^x  \varphi^{(j)}_i(s)\ud s\right)\ud x<\infty $ for some $k\geqslant 2$;
\\
{\bf c.}  $\Phi^{(j)}_i(t)$ are bounded in some neighbourhood of zero;
\\
{\bf d.} For all $t\in[0;\infty)$ the inequalities $\varphi^{(1)}_i(t)\leqslant \lambda_i(t)\leqslant \Phi^{(1)}_i(t)$ and $\varphi^{(2)}_i(t)\leqslant \mu_i(t)\leqslant \Phi^{(2)}_i(t)$ are satisfied $(i\in\{1;2\})$.

If the conditions {\bf a}, {\bf b}, {\bf c} and {\bf d} are satisfied, then the process $X_t$ defined in the Example \ref{ExRT} is quasi-regenerative.
\TR
\end{thm}

\begin{proof}
Denote the d.f. of $j$-th work periods of $i$-th element $\xi_i^{(j)}$ by $F_i^{(j)}(t)$, and denote the d.f. of $j$-th repair periods of $i$-th element $\eta_i^{(j)}$ by $G_i^{(j)}(t)$; $(i\in\{1;2\}, j\in\mathbb Z_{\geqslant 0})$.

Easy to see that $m_i^{(\ell)}\leqslant \EE\,\left(\xi_i^{(j)}\right)^\ell\leqslant M_i^{(\ell)}$ for all $\ell\in[0;k]$, where 
$$
m_i^{(\ell)}=\dd\intl_0^\infty x^{\ell-1}  \exp\left(-\intl_0^x  \Phi^{(j)}_i(s)\ud s\right)\ud x,\mbox{ and} M_i^{(\ell)}=\dd\intl_0^\infty x^{\ell-1}  \exp\left(-\intl_0^x  \varphi^{(j)}_i(s)\ud s\right)\ud x.
$$

To simplify the proof and shorten the text of the paper, we will assume here that $T_i^{(j)}=T$, $\varphi^{(j)}_i(t)=\varphi(t)$, and  $\Phi^{(j)}_i(t)=\Phi(t)$. 
Accordingly, denote $m_i^{(\ell)}=m_\ell$ and $M^{(\ell)}_i=M_\ell$.

Consider the process $X_t$ which starts from the state $X_0=(0,0);(0,0)$, i.e. at the time $t$ both elements start working period.

At the time of the first failure of the first element $t_t$ (see Fig. \ref{ExRT}), the second element can be in the working state or in the repair state.
Its elapsed time in this state can be bounded by the inequality (\ref{GLE}).

The time $t_1$ is the Markov time of the process $X_t$, and this time we will start changing process $X_t$ in such a way that {\it with some non-zero probability}, the beginning of the next period of work of the first element coincides with the beginning of the work of the second element, i.e. to get the process $X_t$ back to state $(0,0);(0,0)$.

Take any number $\Theta>\Xi\bd M_1+\dd\frac{M_2}{2m_1}$. 
Consider two scenarios.

1. At the time $t_1$, the second element is in the repair period, and elapsed time of this period is $B_{t_1}^{(2)}$.
By inequality (\ref{GLE}) we have $\EE\,B_{t_1}^{(2)}<\Xi$, and by Markov inequality we have $\PP\{B_{t_1}^{(2)}>\Theta\}\leqslant \dd\frac{\EE\,B_{t_1}^{(2)}}{\Theta}=\frac{\Xi}{\Theta}$, so, $\PP\{B_{t_1}^{(2)}<\Theta\}\geqslant q_1\bd 1-\dd \frac{\Xi}{\Theta} $.

If event $\{B_{t_1}^{(2)}<\Theta\}$ happened, then at the time $t_1$ both elements are in repair condition: the first one starts the repair period, and the second element continues the repair period, having elapsed repair time less than $\Theta$.
Thus, on some probability space, we can create the remaining time $W_{t_2}^{(2)}$ of the repair of the second element and the total time $\eta_1$ of repair of the first element in such a way that they end at the same time with a probability greater than $\varkappa(\Theta)\bd\inf_{a\in[0;\Theta]}\dd \intl_0^\infty \min\{\varphi(s),\varphi(s+a)\}\ud s$, because the intensities of repair are bounded from below by the function $\varphi(s)$.

2. At the time $t_1$, the second element is in the working period, elapsed time of this period is $B_{t_1}^{(2)}$ and the remaining time of this period is $W_{t_1}^{(2)}$.
By inequality (\ref{GLE1}) we have $\EE\,W_{t_1}^{(2)}<\Xi$, and by Markov inequality we have $\PP\{W_{t_1}^{(2)}>\Theta\}\leqslant \dd\frac{\EE\,W_{t_1}^{(2)}}{\Theta}=\frac{\Xi}{\Theta}$, so, $\PP\{W_{t_1}^{(2)}<\Theta\}\geqslant q_1\bd 1-\dd \frac{\Xi}{\Theta} $.
At the time $t_1$, the first element begins to be repaired, and the repair time $\eta_1$ has a d.f. $F(s)$; denote $q_2\bd \PP\{\eta_1\leqslant\Theta\}=F(\Theta)$.

If events $W_{t_1}^{(2)}>\Theta$ and $\eta_1<\Theta$ have occurred, then at time $t_2\bd t_1+\eta_1$ both elements are in working condition: the first one starts the working period, and the second element continues the working period, having elapsed working time less than $\Theta$.

Thus, on some probability space, we can create the remaining time $W_{t_2}^{(2)}$ of the work of the second element and the total time $\xi_2$ of work of the first element in such a way that they end at the same time with a probability greater than $\varkappa(\Theta)\bd\inf_{a\in[0;\Theta]}\dd \intl_0^\infty \min\{\varphi(s),\varphi(s+a)\}\ud s$>0.

Because $\varphi(s)>0$ a.s. for $s>T$, $\varkappa(\Theta)>0$.

3. If no coupling has occurred, then the described procedure can be repeated at the end of the next working periods of the first element.
Obviously, with probability 1, the constructed process will return to state $((0,0);(0,0))$.

Thereby, we can construct such a continuation of a process $X_t$ that at the time when the repair of the first element is completed, the repair of the second element will be completed, i.e. process $X_t$ will end up in state $((0,0);(0,0))$ from which it started -- with probability greater than $\varkappa(\Theta) q_1q_2$.
It can be seen from the construction procedure that the process remains Markov, and its marginal distributions coincide with the marginal distributions of the original process.

Hence, process $_t$ is quasi-regenerative.
\end{proof}

The construction of ``regenerative copy'' for reliability non-regenerative process can be used for the proof of ergodicity of this process (see, e.g., \cite{ZverkinaCN19}), and then the coupling method applied to the ``regenerative copy'' of the reliability process was used for finding an upper bound of the convergence rate of the distribution of the original reliability process.

The natural generalization of such a problem is a generalization of the notion of the Markov Modulated Poisson Process, which is described in \cite{ZverkinaDCCN20}. 
A ``regenerative copy'' was created for this non-regenerative process, the ergodicity was proved, and then an upper bound for the convergence rate was found. 

\section*{Acknowledgments}
The author is grateful to E.~Yu.~Kalimulina for her valuable help in preparing the text.


\end{document}